\documentclass[12pt]{amsart}

\newtheorem{theorem}{Theorem}[section]

\newtheorem{lemma}[theorem]{Lemma}

\theoremstyle{definition}

\newtheorem{definition}[theorem]{Definition}

\newtheorem{remark}[theorem]{Remark}

\def\R{\mathbb{R}}

\begin{document}

\title[]{Lower bounds for the centered Hardy-Littlewood maximal operator on the real line}

\author{F.J. P\'erez L\'azaro}
\address{Departamento de Matem\'aticas y Computaci\'on,
Universidad  de La Rioja, 26006 Logro\~no, La Rioja, Spain.}
\email{javier.perezl@unirioja.es}


\thanks{The author was partially supported by the Spanish Research Grant with reference PGC2018-096504-B-C32.}

%\thanks{Supported by DGICYT 1317 DP Spain}

%\subjclass{}

%\keywords{}

%\date{}

%\dedicatory{}

%\commby{}

%%% ----------------------------------------------------------------------
\begin{abstract} Let $1<p<\infty$. We prove that there exists an $\varepsilon_p>0$ such that for each $f\in L^p(\mathbb{R})$, the centered Hardy-Littlewood maximal operator $M$ on $\mathbb{R}$ satisfies the lower bound $\|Mf\|_{L^p(\mathbb{R})}\ge (1+\varepsilon_p)\|f\|_{L^p(\mathbb{R})}$.
\end{abstract}

%%% ----------------------------------------------------------------------

\maketitle

%%% ----------------------------------------------------------------------

\section {Introduction}

\markboth{F.J. P\'erez-L\'azaro}{On lower bounds for the centered Hardy-Littlewood maximal operator on the real line}

Given a locally integrable real-valued function $f$ on $\mathbb{R}^d$ define its uncentered maximal function $M_uf(x)$ as
\begin{equation*}
M_uf(x)=\sup_{B\ni x}\frac{1}{|B|}\int_B|f(y)|dy,
\end{equation*}
where the supremum is taken over all balls $B\in\mathbb{R}^d$ containing the point $x$; here $|B|$ denotes the $d$-dimensional Lebesgue measure of the ball $B$. The usefulness of this and other maximal functions comes from the fact that they are larger than the original function $f$, but not much larger, and usually improve regularity. Since $M_uf$ is often used as a close upper bound for $f$, it is interesting to know precisely how much larger $M_uf$ is, and the same question can be asked about other maximal operators.

It is well known that $M_uf(x)\ge f(x)$ a.e. On the other hand, since an average does not exceed a supremum, $\|M_uf\|_{L^\infty(\mathbb{R}^d)}=\|f\|_{L^\infty(\mathbb{R}^d)}$.
It is shown in  \cite{MaSo} that 
 $M_u$ has no nonconstant fixed points.  In \cite{Ler} A. Lerner studied  whether given any $1<p<\infty$, there is a constant $\varepsilon_{p,d}>0$ such that
\begin{equation}\label{quesLer}
\|M_u f\|_{L^p(\mathbb{R}^d)}\ge (1+\varepsilon_{p,d})\|f\|_{L^p(\mathbb{R}^d)}\quad\text{for all }f\in L^p(\mathbb{R}^d).
\end{equation}
We note that lack of existence of nonconstant fixed points does not imply (\ref{quesLer}).
Using Riesz's sunrise lemma, Lerner proved for the real line that
\begin{equation*}
\|M_uf\|_{L^p(\mathbb{R})}\ge \left(\frac{p}{p-1}\right)^{1/p}\|f\|_{L^p(\mathbb{R})}.
\end{equation*}  
A proof of inequality (\ref{quesLer}) for every dimension $d\ge 1$ and every $1<p<\infty$ was obtained in \cite{IJN}. Inequality (\ref{quesLer}) has been shown to be true for other maximal functions, say, maximal funtions defined taking the supremum over shifts and dilates of a fixed centrally symmetric convex body, maximal functions defined over $\lambda$-\textit{dense family} of sets, \textit{almost centered} maximal functions (see \cite{IJN}) and dyadic maximal functions \cite{MeNi}.

For the centered maximal function
\begin{equation*}
Mf(x)=\sup_{r>0}\frac{1}{|B_r(x)|}\int_{B_r(x)}|f(y)|dy,
\end{equation*} 
Lerner's inequality 
\begin{equation}\label{quescenter}
\|M f\|_{L^p(\mathbb{R}^d)}\ge (1+\varepsilon_{p,d})\|f\|_{L^p(\mathbb{R}^d)}\quad\text{for all }f\in L^p(\mathbb{R}^d),
\end{equation}
need not hold. First of all, it was shown in \cite{Ko} that $M$ has a nonconstant fixed point $f\in L^p(\mathbb{R}^d)$ (that is, $Mf=f$) if and only if $d\ge 3$ and $p>d/(d-2)$. But, as was noted before, the lack of nonconstant fixed points does not imply (\ref{quescenter}). In this context, Ivanisvili and Zbarsky (cf. \cite{IZ}) noted that (\ref{quescenter}) is valid for any $d$ when $p\equiv p_d$ is sufficiently close to $1$.

The main result in \cite{IZ} proves for $d=1$ and every $1<p<2$ that (\ref{quescenter}) is true, in the form
\begin{equation}\label{IvanZba}
\|Mf\|_{L^p(\mathbb{R})}\ge \left(\frac{p}{2(p-1)}\right)^{1/p}\|f\|_{L^p(\mathbb{R})}.
\end{equation}  
They also proved that inequality (\ref{quescenter}) holds for $d=1$ and $1<p<\infty$, if we restrict $f$ to the class of indicator functions or unimodal functions. Besides, they conjectured (see \cite[p. 343]{IZ}) that (\ref{quescenter}) is valid for $d=1$ and $1<p<\infty$ without restrictions on the functions. 

In this paper we give an afirmative answer to their conjecture, proving the following
\begin{theorem}\label{theo_main}
Let $1<p<\infty$. Then there exists an $\varepsilon_p>0$ such that
\begin{equation*}
\|Mf\|_{L^p(\mathbb{R})}\ge (1+\varepsilon_p)\|f\|_{L^p(\mathbb{R})}\quad \text{for any }f\in L^p(\mathbb{R}).
\end{equation*}
Furthermore, if $A_p$ is the best constant for the strong $(p,p)$ inequality satisfied by the centered maximal operator on the real line, and $\gamma_n$ is as in Definition \ref{def_gamma}, then for every $n\ge 1$ we can select
\begin{equation*}
(1+\varepsilon_p)^p=1+\left(\frac{A_p-1}{A_p^n-1}\right)^p\left[\left(\frac{\gamma_n p}{(p-1)}\right)^{1/p}-1\right]^p.
\end{equation*}
\end{theorem}
Let us note that this expression is stricly larger that $1$ if we suitably choose $n$, taking into account that $\gamma_n \uparrow 1$ (see Remark \ref{rem_gamma}).

Our approach consists of extending the methods in \cite{IZ} and using the following inequality (see Lemma \ref{lem_Mn} below) for any locally integrable function in $\mathbb{R}$:
\begin{equation}\label{des_approach}
M^nf\ge \gamma_n   M_Lf,
\end{equation}
where $M_L$ denotes the left maximal operator and $M^n$ denotes the iteration of the centered maximal operator $n$ times. This inequality extends the trivial inequality $Mf\ge M_Lf/2$.

Using (\ref{des_approach}), we prove
\begin{theorem}\label{theo_Mn}
Let $n\in\mathbb{N}$ and $f\in L^p(\R)$. Then,
\begin{equation*}
\|M^nf\|_p\ge \left(\frac{\gamma_n p}{(p-1)}\right)^{1/p}\|f\|_p.
\end{equation*}
\end{theorem}
Since $\gamma_1=1/2$, this result is an extension of (\ref{IvanZba}).

Let us remark that simultaneously and independently, Zbarsky \cite{Zb} has proved (\ref{quescenter}) for $d=1$ and $d=2$ and the centered maximal operator associated to centrally symmetric convex bodies. This extends Theorem \ref{theo_main}, but without an explicit expression for the lower constant $\varepsilon_p$. 
\vskip .3 cm

I am indebted to Prof. J. M. Aldaz for some suggestions that improved the presentation of this note.

\section{Definitions and lemmas}
 
 \begin{definition}\label{def_g_functions}
For all $n\in\mathbb{N}\cup\{0\}$, define the following functions $g_n:[-1/2,\infty)\longrightarrow [0,1]$.
Let $g_0$ be the null function and for $n\ge 1$, set
\begin{equation}\label{eqrecursive}
g_n(t):=\frac{1+\int_0^{1+2t}g_{n-1}(u)du}{2(1+t)},\quad t\ge -\frac{1}{2}.
\end{equation} 
\end{definition}

In the next lemma we give an explicit formula for the functions $g_n$.

\begin{lemma}\label{lem_g_explicit}
Let $\{g_n\}_{n=0}^\infty$ be the functions from Definition \ref{def_g_functions}. Then,
\begin{enumerate}
\item $0\le g_n(t)\le 1$ for all $n\in\{0\}\cup\mathbb{N}$ and all $t\ge -1/2$.
\item For all $n\ge 0$ and $t\ge -1/2$, we have
\begin{equation*}
g_n(t)=\frac{\log(2+2t)}{1+t}\sum_{j=1}^{n}\frac{\log^{j-2}(2^{j}(1+t))}{2^{j} (j-1)!}.
\end{equation*}
\item For all $t\ge -1/2$, we have $\lim_{n\to\infty}g_n(t)=1$.
\end{enumerate}
\end{lemma}

\begin{proof}
Part 1 of the lemma follows by simple induction in $n$.
To prove part 2, for each $n\in \{0\}\cup\mathbb{N}$, we define $h_n(t):=g_{n+1}(t)-g_n(t)$. Since $g_0(t)=0$, it holds that
\begin{equation*}
g_{n}(t)=\sum_{j=0}^{n-1}h_j(t).
\end{equation*}

Let us note that $h_0(t)=g_1(t)-g_0(t)=g_1(t)=1/(2+2t)>0$. Besides, by (\ref{eqrecursive}),
\begin{equation}\label{eqhrecursive}
h_n(t)=\frac{\int_0^{1+2t}h_{n-1}(u)du}{2(1+t)},\quad n\ge 1,\,t\ge -1/2. 
\end{equation}

Now we set for each $n\ge 0$, 
\begin{equation*}
c_n(t)=\frac{\log(2+2t)}{1+t}\frac{1}{2^{n+1} n!}\log^{n-1}(2^{n+1}(1+t)),\quad t\ge -1/2,
\end{equation*}
where $c_0(-1/2)$ is defined by continuity, i.e., $c_0(-1/2)=\lim_{t\to -1/2^+}c_0(t)=\lim_{t\to -1/2^+}1/(2+2t)=1$. Then we have that $h_0(t)=c_0(t)$ for all $t\in[-1/2,\infty)$. Moreover, it is a calculus exercise to check that (\ref{eqhrecursive}) also holds with $c_n$ and $c_{n-1}$ instead of $h_n$ and $h_{n-1}$, for every $n\ge 1$. As a consequence, we have that $c_n(t)=h_n(t)$ for all $n\ge 0$ and all $t\in[-1/2,\infty)$. Thus, part 2 of the lemma holds.

Finally we will prove part 3 of the lemma. By part 2, we have that
\begin{equation}\label{eqlimit}
\lim_{n\to\infty}g_n(t)= \frac{\log(2+2t)}{1+t}\sum_{j=1}^{\infty}\frac{\log^{j-2}(2^{j}(1+t))}{2^{j} (j-1)!},\quad t\ge -1/2.
\end{equation}

 As a consequence of Lagrange expansion \cite[p.206, eq.6.24]{Ch}, we can obtain
\begin{equation*}
e^{xy}=\sum_{k=0}^\infty x(x+kz)^{k-1}\frac{(y e^{-y z})^k}{k!},\quad x,y,z\in\mathbb{R}.
\end{equation*}
This equation with $y=1$, $x=\log(2+2t)$ and $z=\log2$ implies
\begin{equation*}
2+2t=\sum_{k=0}^\infty \log(2+2t)(\log(2+2t)+k\log2)^{k-1}\frac{2^{-k}}{k!}.
\end{equation*}
From this equation and (\ref{eqlimit}), part 3 of the lemma follows.
\end{proof}

\begin{remark}
Part 3 of the lemma could also be proved by  suitably bounding the functions $g_n$, using an argument inspired in \cite[p.4-5]{IZ} to obtain:
\begin{equation}\label{eqboundinductive}
1\ge g_n(t)\ge 1-\left(\frac{\sqrt{8}}{3}\right)^n\sqrt{1+t},\quad n\in\mathbb{N}, t\ge 0.
\end{equation}
%Indeed, it can be seen that inequality (\ref{eqboundinductive}) is true for $n=1$. Next we assume that (\ref{eqboundinductive}) is true for $n-1$ and, by this assumption and (\ref{eqrecursive}),
%\begin{equation*}
%g_n(t)=\frac{1+\int_0^{1+2t}g_{n-1}(u)du}{2(1+t)}\ge \frac{\int_{-1}^{1+2t}\left(1-\left(\frac{\sqrt{8}}{3}\right)^{n-1}\sqrt{u+1}\right)du}{2(1+t)}=
%\end{equation*}
%\begin{equation*}
%=1-\left(\frac{\sqrt{8}}{3}\right)^{n-1}\frac{\int_{-1}^{1+2t}\sqrt{u+1}du}{2(1+t)}
%=1-\left(\frac{\sqrt{8}}{3}\right)^{n-1}\left(\frac{\sqrt{8}}{3}\right)\sqrt{1+t}.
%\end{equation*} 
%And (\ref{eqboundinductive}) is proved. Now, (\ref{eqboundinductive}) implies inmmediately that $g_n(t)$ tends to $1$ when $n$ tends to infinity.
\end{remark}

\begin{definition}\label{def_gamma}
For each $n\in \mathbb{N}$, let us denote by $\gamma_n:=g_n(0)$, where $g_n$ are the functions from Definition \ref{def_g_functions}. 
\end{definition}

\begin{remark}\label{rem_gamma}
It follows from the previous definition and Lemma \ref{lem_g_explicit} that
\begin{equation*}
\gamma_n=\frac{1}{2}\sum_{j=1}^n \frac{j^{j-2}}{(j-1)!}\left(\frac{\log2}{2}\right)^{j-1},\quad n\in \mathbb{N}.
\end{equation*}
Besides, $\gamma_1=1/2$ and $\gamma_n$ increases to $1$ when $n\to \infty$.
\end{remark}

\begin{definition}
Let $f:\mathbb{R}\longrightarrow\mathbb{R}$ be a locally integrable function. We define the left maximal function $M_Lf$ as
\begin{equation*}
M_Lf(x)=\sup_{h>0}\frac{1}{h}\int_{x-h}^x |f(u)|du,\quad x\in\mathbb{R}.
\end{equation*}
\end{definition}

It is easy to see that $Mf\ge M_Lf/2$. In the next lemma we  extend this inequality to the iterated centered maximal operator, defined via  $M^1f:=Mf$, and for $n\ge 2$, $M^nf:=M(M^{n-1}f)$.

\begin{lemma}\label{lem_Mn}
Let $\{\gamma_n\}_{n=1}^\infty$ be the sequence from Definition \ref{def_gamma}. Then, for all $n\in\mathbb{N}$ and all $f$ in $L^1_{loc}(\R)$, $M^nf \ge \gamma_n M_Lf$.  
\end{lemma}
\begin{proof}
Let us assume that $f\ge 0$. Fix $x\in \R$ and $h>0$. Define $F(x,h):=\frac{1}{h}\int_{x-h}^xf(t)dt$. Now, using an inductive process in $n\in\mathbb{N}$, we will prove that for all $y\ge x$, 
\begin{equation}\label{recursive}
M^nf(y)\ge F(x,h) \ g_n\left(\frac{y-x}{h}\right),
\end{equation}
where $g_n$ comes from Definition \ref{def_g_functions}.
Indeed, for $n=1$ and every $y\ge x$, we have 
\begin{equation*}
Mf(y)\ge \frac{1}{2(y-x+h)}\int_{x-h}^{2y-x+h}f(t)dt\ge  \frac{hF(x,h)}{2(y-x+h)}= \frac{F(x,h)}{2\left(1+\frac{y-x}{h}\right)}=F(x,h)g_1\left(\frac{y-x}{h}\right).
\end{equation*}
Hence, by induction hypothesis, for all $n\ge 2$ and all $y\ge x$,
\begin{equation*}
M^nf(y)\ge \frac{1}{2(y-x+h)}\int_{x-h}^{2y-x+h}M^{n-1}f(t)dt\ge  \frac{\int_{x-h}^x f(t)dt+\int_{x}^{2y-x+h}M^{n-1}f(t)dt}{2(y-x+h)}\ge
\end{equation*}
\begin{equation*}
\frac{hF(x+h)+F(x,h)\int_{x}^{2y-x+h}g_{n-1}\left(\frac{t-x}{h}\right)dt}{2(y-x+h)}
=hF(x,h)\frac{1+\int_{0}^{1+2\frac{y-x}{h}}g_{n-1}(z)dz}{2(y-x+h)}=
\end{equation*}
\begin{equation*}
=F(x,h)\frac{1+\int_{0}^{1+2\frac{y-x}{h}}g_{n-1}(z)dz}{2(1+\frac{y-x}{h})}=F(x,h)g_n\left(\frac{y-x}{h}\right),
\end{equation*}
so (\ref{recursive}) is proved. As a consequence, $M^nf(x)\ge F(x,h)g_n(0)=F(x,h)\gamma_n$, and taking the supremum over $h>0$ we obtain
\begin{equation*}
M^nf(x)\ge M_Lf(x)\cdot \gamma_n,\quad n\in\mathbb{N}.
\end{equation*} 
\end{proof}

\begin{remark}
It is known that for every $f\in L^p(\mathbb{R})$, we have $\|M_Lf\|_p\ge \left(\frac{p}{p-1}\right)^{1/p}\|f\|_p$ (see \cite[p.93, 2.1.11(a)]{Gra} and integrate).
This inequality, together with the previous lemma, leads inmediately to
\begin{equation}\label{ineq_weak}
\|M^nf\|_p\ge \gamma_n\|M_Lf\|_p\ge \gamma_n\left(\frac{p}{p-1}\right)^{1/p}\|f\|_p.
\end{equation}
This inequality is enough to prove Theorem \ref{theo_main}, but with a smaller $\varepsilon_p$. Indeed, inequality (\ref{ineq_weak}) will be improved in Theorem \ref{theo_Mn} by the use of the following lemma, which is an extension of \cite[Lemma 3]{IZ} and uses the same arguments. We include it here for the reader's convenience.  
\end{remark}

\begin{lemma}\label{lem_last}
Let $0<\lambda<\infty$ and $n\in\mathbb{N}$. For every locally integrable function $f\ge 0$ defined on the real line, it holds that 
\begin{equation*}
|\{M^nf>\lambda\}|\ge \frac{\gamma_n}{\lambda}\int_{\{f>\lambda\}}f.
\end{equation*}
\end{lemma}
\begin{proof}
Since $M^nf\ge f$ almost everywhere and, by Lemma \ref{lem_Mn}, $M^nf\ge \gamma_n M_Lf$, we have (with the exception of a null set) that  
\begin{equation*}
\{M^nf>\lambda\}\supseteq \{f>\lambda\}\cup \{M_Lf>\frac{\lambda}{\gamma_n}\}.
\end{equation*}
Then, we separate this into two disjoint sets, take Lebesgue measure, apply $M_Lf>f$ a.e. and using Riesz's rising sun lemma \cite[p.93]{Gra} we obtain,
\begin{equation*}
|\{M^nf>\lambda\}|\ge |\{f>\lambda\}\setminus \{M_Lf>\frac{\lambda}{\gamma_n}\}|+| \{M_Lf>\frac{\lambda}{\gamma_n}\}|\ge
\end{equation*}
\begin{equation*}
\ge \frac{\gamma_n}{\lambda}\int_{\{f>\lambda\}\setminus \{M_Lf>\frac{\lambda}{\gamma_n}\}}f+\frac{\gamma_n}{\lambda}\int_{\{M_Lf>\frac{\lambda}{\gamma_n}\}}f\ge \frac{\gamma_n}{\lambda}\int_{\{f>\lambda\}\cup\{M_Lf>\frac{\lambda}{\gamma_n}\}}f\ge\frac{\gamma_n}{\lambda}\int_{\{f>\lambda\}}f.
\end{equation*}
\end{proof}
\section{Proofs of the theorems}
To prove the theorems one just has to use the previous lemmas and some arguments from \cite{IZ}. We include here the proofs for the reader's convenience.

%\begin{theorem}\label{theo_Mn}
%Let $n\in\mathbb{N}$ and $f\in L^p(\R)$. Then,
%\begin{equation*}
%\|M^nf\|_p\ge \left(\frac{\gamma_n p}{(p-1)}\right)^{1/p}\|f\|_p.
%\end{equation*}
%\end{theorem}
\begin{proof}[Proof of Theorem \ref{theo_Mn}]
Without loss of generality we assume that $f\ge 0$. By Lemma \ref{lem_last} we have:
\begin{equation*}
|\{M^nf>\lambda\}|\ge \frac{\gamma_n}{\lambda}\int_\mathbb{R}f(x)\chi_{(\lambda,\infty)}(f(x))dx.
\end{equation*}
We multiply both sides of the previous inequality by $p\lambda^{p-1}$ and integrate:
\begin{equation*}
\int_\mathbb{R}(M^nf(x))^pdx\ge \int_0^\infty \gamma_n p\lambda^{p-2}\int_\mathbb{R}f(x)\chi_{(\lambda,\infty)}(f(x))dxd\lambda=
\end{equation*}
\begin{equation*}
\gamma_n p\int_\mathbb{R} f(x)\int_0^{f(x)}\lambda^{p-2} d\lambda dx=\frac{\gamma_n p}{(p-1)}\int_\mathbb{R} f(x)^p dx.
\end{equation*}
\end{proof}

%\begin{theorem}
%Let $A_p$ be the best constant for strong $(p,p)$ inequality for the centered maximal operator on the real line. Let $\gamma_n$ be defined as in Definition \ref{def_gamma}. Then
%\begin{equation*}
%\|Mf\|_p\ge \|f\|_p \left\{1+\left(\frac{A_p-1}{A_p^n-1}\right)^p\left[\left(\frac{\gamma_n p}{(p-1)}\right)^{1/p}-1\right]^p\right\}^{1/p}.
%\end{equation*}
%\end{theorem}
\begin{proof}[Proof of Theorem \ref{theo_main}]
First, we have that
\begin{equation}\label{eqA}
\|Mf\|_p^p\ge \|f\|_p^p+\|Mf-f\|_p^p.
\end{equation}
Besides, if we denote by $A_p>1$ the best constant for the strong $(p,p)$ inequality satisfied by $M$, it holds that 
\begin{equation}\label{eqB}
\|M^nf-f\|_p\le \sum_{i=1}^n\|M^if-M^{i-1}f\|_p\le \sum_{i=1}^n A_p^{i-1}\|Mf-f\|_p=\frac{A_p^n-1}{A_p-1}\|Mf-f\|_p.
\end{equation}
Furthermore, by Theorem \ref{theo_Mn}
\begin{equation*}
\left(\frac{\gamma_n p}{(p-1)}\right)^{1/p}\|f\|_p\le \|M^nf\|_p\le \|M^nf-f\|_p+\|f\|_p .
\end{equation*}
Then
\begin{equation}\label{eqC}
\left[\left(\frac{\gamma_n p}{(p-1)}\right)^{1/p}-1\right]\|f\|_p\le \|M^nf-f\|_p.
\end{equation}
Now, putting (\ref{eqA}), (\ref{eqB}) and (\ref{eqC}) together we get
\begin{equation*}
\|Mf\|_p^p\ge \|f\|_p^p +\left(\frac{A_p-1}{A_p^n-1}\right)^p\|M^nf-f\|_p^p\ge
\end{equation*}
\begin{equation*}
 \|f\|_p^p  + \left(\frac{A_p-1}{A_p^n-1}\right)^p\left[\left(\frac{\gamma_n p}{(p-1)}\right)^{1/p}-1\right]^p\|f\|_p^p=
\end{equation*}
\begin{equation*} 
= \|f\|_p^p \left\{1+\left(\frac{A_p-1}{A_p^n-1}\right)^p\left[\left(\frac{\gamma_n p}{(p-1)}\right)^{1/p}-1\right]^p\right\}.
\end{equation*}
Let us note that, by Remark \ref{rem_gamma}, for $n$ big enough, $\gamma_n p/(p-1)>1$. 
\end{proof}

\end{document}